\documentclass[12pt, reqno]{amsart}
\usepackage{amsmath, amsthm, amscd, amsfonts, amssymb, graphicx, color, amssymb, mathrsfs}
\usepackage{enumerate}
\usepackage[bookmarksnumbered, colorlinks, plainpages]{hyperref}
\hypersetup{colorlinks=true,linkcolor=red, anchorcolor=green, citecolor=cyan, urlcolor=red, filecolor=magenta, pdftoolbar=true}

\newtheorem{theorem}{Theorem}[section]
\newtheorem{lemma}[theorem]{Lemma}
\newtheorem{proposition}[theorem]{Proposition}

\theoremstyle{definition}
\newtheorem{definition}[theorem]{Definition}

\theoremstyle{remark}
\newtheorem{remark}[theorem]{Remark}
\numberwithin{equation}{section}

\begin{document}
\setcounter{page}{1}

\title[Hilbert
transforms along curves]{Vector-valued Hilbert transforms along curves}

\author[G. Hong, H. Liu]{Guixiang Hong$^1$ and Honghai Liu$^2$$^{*}$}

\address{$^{1}$School of Mathematics and Statistics,
Wuhan University,
Wuhan 430072, China
and
 Instituto de Ciencias Matem\'aticas,
CSIC-UAM-UC3M-UCM,
Consejo Superior de Investigaciones Cient\'ificas,
C/ Nicol\'{a}s Cabrera 13-15. 28049, Madrid, Spain.}
\email{\textcolor[rgb]{0.00,0.00,0.84}{guixiang.hong@icmat.es}}

\address{$^{2}$ School of Mathematics and Information Science, Henan Polytechnic University, Jiaozuo, Henan 454003, China.}
\email{\textcolor[rgb]{0.00,0.00,0.84}{hhliu@hpu.edu.cn}}




\date{Received: xxxxxx; Revised: yyyyyy; Accepted: zzzzzz.
\newline \indent $^{*}$ Corresponding author}

\begin{abstract}
In this paper, we show that Hilbert transforms along some curves are
bounded on $L^p({\mathbb R}^n;X)$ for some $1<p<\infty$ and some UMD
spaces $X$. In particular, we prove that Hilbert transforms along
some curves are completely $L^p$-bounded in the terminology from
operator space theory. Moreover, we obtain the
$L^p(\mathbb{R}^n;X)$-boundedness of anisotropic singular integrals
by using the "method of rotations" of Calder\'{o}n-Zygmund.
 All these results extend already existing related ones.
\end{abstract} \maketitle

\section{Introduction}\label{sect1}
The question of whether the mapping properties of singular integral operators could be extended to the Lebesgue-B\^ohner spaces $L^p(\mathbb{R}^n;X)$ ($1<p<\infty$) of vector-valued functions was taken up by several authors in the 60's. In \cite{BCP62}, Benedek, Calder\'on and Panzone observed that the boundedness on $L^{p_0}(\mathbb{R}^n;X)$ for one $1<p_0<\infty$ of a singular integral operator, together with H\"ormander's condition, implies its
boundedness on $L^p(\mathbb{R}^n;X)$ for all $1<p<\infty$. However, to actually get the $L^{p_0}(\mathbb{R}^n;X)$-boundedness (something that was immediate for $p_0=2$ in the scalar-valued), turned out to be a significantly difficult task except in the case $X=L^{p_0}(\Omega)$ for some measure space $\Omega$.

The first progress made in this direction is Burkholder's extension
\cite{Bu83}  of Riesz's classical theorem on the $L^p$-boundedness
of the Hilbert transform, where it was shown that if the underlying
Banach space $X$ satisfies the so called UMD-property, then the
Hilbert transform is bounded on $L^p(\mathbb{R};X)$ for any
$1<p<\infty$. Moreover, the UMD-property was shown by Bourgain
\cite{Bou86} to be necessary for the boundedness of the Hilbert
transform. It is well-known that the Hilbert transform is a
prototype of singular integral operators and Fourier multipliers,
its boundedness motivates McConnell's \cite{Mcc84} and Zimmermann's
\cite{Z89} results on vector-valued Marcinkiewicz-Mihlin
multipliers, and Hyt\"{o}nen and Weis's \cite{HW06}
results on vector-valued singular convolution integrals.

Particularly, if $X$  equals $S_p$--the Schatten class, the
$L^p(\mathbb R^n;S_p)$-boundedness is called complete
$L^p$-boundedness in the light of noncommutative harmonic analysis.
In this setting, the complete $L^2$-boundedness is immediately
available because $S_2$ is a Hilbert space, and the Fourier transform
(or almost orthogonality principle) can be adapted. In order to
obtain the complete $L^p$-boundedness, so far as we know in the
noncommutative harmonic analysis, there are only two ways. One way
is to establish firstly the weak type $(1,1)$ estimate, and then to
use interpolation and the duality argument. In this way, the convolution
kernel need to satisfy the Lipschitz regularity in order to conduct
the pseudo-localization principle as done in \cite{Par09} (see also
\cite{HSMP14} for related results). The other way is to get
$(L^{\infty},BMO)$ (the noncommutative BMO space) estimate, then to
use interpolation and the duality argument. In this case, the kernel is
required to satisfy the H\"ormander's condition as done in
\cite{Mei07} and \cite{JMP}. However, to get the complete
$L^p$-boundedness is not a trivial work when the kernel does not
satisfy the Lipschitz regularity and the H\"ormander condition, see
e.g. \cite{HoPa} for more information.

The purpose of our project is to extend the vector-valued
singular integrals theory to more general setting. We consider
vector-valued singular Radon transforms, which are given by the
following principal-valued integral
\begin{equation}\nonumber
\mathscr Tf(x)={\rm p.v.}\int_{\mathbb R^k}f(x-\Gamma(t))K(t)dt,\ \
f\in C_0^\infty({\mathbb R}^n)\otimes X,
\end{equation}
where $X$ is a Banach space, $K$ is a Calder\'{o}n-Zygmund kernel in
$\mathbb R^k$ and $\Gamma:\mathbb R^k\rightarrow{\mathbb R}^n$ is a
surface in ${\mathbb R}^n$ with $\Gamma(0)=0$, $n\ge2$. Precisely,
we are interested in the boundedness of $\mathscr T$ on $L^p(\mathbb
R^n;X)$, where $p\in(1,\infty)$ and $X$ is some Banach space.
Obviously, $\mathscr T$ are classical vector-valued singular
convolution integrals if $k=n$ and $\Gamma(t)=(t_1,t_2,\cdots,t_n)$, and
related results have been introduced in the previous paragraphs. On
the other hand, if $X=\mathbb R$, $\mathscr T$ are classical
singular integrals associated to surfaces, which have been
well-studied by Stein, Nagel, Wainger, Christ and so on, see
\cite{SW78} for a survey of results through 1978 and \cite{CNSW99}
through 1999.
\par
In the present paper, we start with the investigation of Hilbert transforms along curves in the hope of providing the insight and inspiration for subsequent development of this subject, as the role played by the classical Hilbert transform in the classical vector-valued Calder\'{o}n-Zygmund theory. Vector-valued Hilbert transforms along curves are defined by
\begin{equation}\nonumber
\mathscr{H}f(x)={\rm p.v.}\int_{\mathbb
R}f\big(x-\Gamma(t)\big)\frac{dt}t,\ \ f\in C_0^\infty({\mathbb
R}^n)\otimes X.
\end{equation}
In the scalar-valued case, the $L^2$-boundedness goes back the work
\cite{Fab66} of Fabes who proved it with
$\Gamma(t)=(t^{\alpha},t^{\beta})$ using complex integration. Then
Stein and Wainger \cite{StWa70} obtained the $L^2$-boundedness for
all homogeneous curves by using Van der Corput's estimates for
trigonometric integrals. The first breakthrough was the proof of the
$L^p$-boundedness in the papers of Nagel, Rivi\`{e}re and Wainger
\cite{NRW76} as well as the paper of Nagel and Wainger
\cite{NW76} using Stein's complex interpolation. Since then, many
related results have been obtained, see Stein and Wainger's survey
paper \cite{SW78} for the curves having some curvature at the
origin, the paper of Carlsson {\it et al} \cite{CCCDRVWW86} and the
references therein for the flat curves in $\mathbb{R}^2$. However, all results about vector-valued singular
integrals mentioned previously can not be directly applied to
Hilbert transforms along curves on $L^p(\mathbb R^n;X)$, because
they are no longer Calder\'on-Zygmund operators. Therefore this
study is a move beyond the vector-valued Calder\'on-Zygmund theory.

In the present paper, we extend Nagel, Rivi\`{e}re and Wainger as
well as Nagel and Wainger's results mentioned above to the
vector-valued setting by combining their original arguments and some
idea developed recently by Hyt\"onen and Weis \cite{HW08} in the
vector-valued Calder\'on-Zygmund theory. To state our results, we
need to recall and introduce some notations. Denote by $\epsilon_j$,
$j\in\mathbb Z$, the Rademacher system of independent random
variables on a probability space $(\Omega,\Sigma,\mathbf P)$
verifying $\mathbf P(\epsilon_j=1)=\mathbf P(\epsilon_j=-1)=1/2$.
Let $\mathbb E=\int(\cdot)d\mathbf P$ be the corresponding
expectation. The main Banach space geometry property of $X$ we are
concerned in this paper is the UMD property (see e.g. \cite{Bu83}),
i.e. the following inequality holds:
$$
\big(\mathbb E\big\|\sum_{k=1}^N\epsilon_kd_k\big\|_{X}^2\big)^{1/2}\le C\big(\mathbb E\big\|\sum_{k=1}^Nd_k\big\|_{X}^2\big)^{1/2}
$$
for all $N\in\mathbb N$, all fixed signs $\epsilon_k\in\{-1,1\}$,
all $X$-valued martingale differences $(d_k)_{k\geq0}$. The
following notation is very useful for formulating the main results
in this paper.
\begin{definition}\label{subclass of umd}
Let  $(a,b)\subseteq(0,1)$. We define $\mathcal I_{(a,b)}$ to be the set consisting of UMD spaces with its element $X$ having the form $X=[H,Y]_\theta$ such that  $\theta\in(a,b)$, $H$ is a Hilbert space and $Y$ is another UMD space.
$\mathcal I_{(0,1)}$ is denoted by $\mathcal{I}$ for simplicity.
\end{definition}
\begin{remark}
{\rm (i)}. It is easy to check that all the noncommutative $L_p$ spaces (containing commutative $L^p$ spaces) with $1<p<\infty$ belong to the class $\mathcal
I_{(|1-\frac2p|,1)}$. From the reflexivity of UMD space, in general we have $X\in\mathcal I_{(a,b)}$ if
and only if $X^\ast\in\mathcal I_{(a,b)}$. Furthermore, if
$(a,b)\subseteq(c,d)\subseteq(0,1)$, then $\mathcal
I_{(a,b)}\subseteq\mathcal I_{(c,d)}$.

{\rm (ii)}. In \cite{R85}, Rubio de Francia proved that for any UMD
lattice $X$ there exist $\theta\in(0,1)$, a Hilbert space $H$ and
another UMD lattice $Y$ such that $X=[H,Y]_\theta$. That means every
UMD lattice $X$ belongs to $\mathcal I$. In the same paper, the
author also ask the open question ``Is every $B\in UMD$ intermediate
between a 'worse' $B_0$ and a Hilbert spaces ?"  which in our
language means ``If $\mathcal I$ contains all UMD spaces?".
\end{remark}

\par
The first result is on the Hilbert transform along the homogeneous curves
$\Gamma(t)=(|t|^{\alpha_1}sgn t,|t|^{\alpha_2}sgn
t,\cdots,|t|^{\alpha_n}sgn t)$ with each $\alpha_i>0$.

\begin{theorem}\label{main result}
Let $X\in\mathcal I$ and $1<p<\infty$. Then there exists an absolute
constant $C_p$ such that
\begin{equation}\nonumber
\|\mathscr Hf\|_{L^p(X)}\le C_p
\|f\|_{L^p(X)},\ \ f\in{L^p(\mathbb R^n;X)} .
\end{equation}
\end{theorem}

This is a vector-valued version of Theorem 1 of Nagel, Rivi\`{e}re
and Wainger in \cite{NRW76}. Following the previous remark, Theorem
\ref{main result} implies the complete boundedness of Hilbert
transforms along this kind of curves which is of independent
interest in the operator space theory. This result also partially
generalize the previous result by Rubio de Francia, Ruiz and Torra
\cite{RRT86} where they obtained Theorem \ref{main result} in the
case $X=\ell^q$ with $1<q<\infty$. In \cite{RRT86}, the authors used
indirectly Benedek, Calder\'on and Panzone's strategy mentioned
previously. While the proof of Theorem \ref{main result} is
motivated by the recent development in the vector-valued
Calder\'on-Zygmund theory \cite{HW06}, see Section 2 for related details.

Let $\delta_t$ be a one parameter group of dilations and $\mathbf
e,\mathbf f$ be vectors in $\mathbb R^n$.
 A curve $\Gamma(t)$ is called two-sided homogeneous
if the following two conditions hold:
\begin{equation}\label{ts}
\Gamma(t)=\left\{
\begin{array}{ccc}
\delta_t\ \mathbf e,& t>0,\\
\delta_{-t}\ \mathbf f,& t<0,\\
0,& t=0;
\end{array}\right.
\end{equation}
\begin{equation}\nonumber
\{\xi|\xi\cdot\Gamma(t)\equiv0,t>0\}=\{\xi|\xi\cdot\Gamma(t)\equiv0,t<0\}.
\end{equation}
The curve  $\Gamma(t)=(|t|^{\alpha_1}sgn t,|t|^{\alpha_2}sgn
t,\cdots,|t|^{\alpha_n}sgn t)$ is a model with
$\delta_tx=(t^{\alpha_1}x_1,t^{\alpha_2}x_2,\cdots,t^{\alpha_n}x_n)$,
$\mathbf e=\mathbf 1$ and $\mathbf f=-\mathbf 1$. We will see that
the same argument for this particular curve works for all the curves
with the same dilation but $\mathbf e=-\mathbf f$. Generalization of
Theorem \ref{main result} to all two-sided homogeneous curves in
turn motivates  us to consider the vector-valued Calder\'on-Zygmund
theory associated to one parameter group of dilations, which is a
project under  progress.

As an application, Theorem \ref{main result} is used to deal
with vector-valued anisotropic singular integrals with homogeneous
kernel by Calder\'on-Zygmund's rotation method. This work improves
Hyt\"{o}nen's Theorem 5.2 in \cite{H07} in some sense, see Section 3
for more details.

In the next result, we deal with certain convex curves in
$\mathbb{R}^2$ with the form $\Gamma(t)=\big(t,\gamma(t)\big)$,
$\gamma(t)$ is some convex function for $t\ge0$.
\begin{theorem}\label{main result 2}
Let $X$ be an UMD lattice belonging to the class
$I_{(0,\frac15)}$, $\gamma(t)$ be a continuous odd function, twice
continuously differentiable, increasing and convex for $t\ge0$.
 Suppose also that $\gamma''$ is monotone for $t>0$ and
there exists $C>0$ so that $\gamma'(t)\le Ct\gamma''(t)$ for $t>0$.
Then for $\frac53<p<\frac52$, there exists an absolute constant
$C_p$ such that
\begin{equation}\nonumber
\|\mathscr Hf\|_{L^p(X)}\le C_p
\|f\|_{L^p(X)},\ \ f\in{L^p(\mathbb R^n;X)} .
\end{equation}
\end{theorem}

 A large class of functions $\gamma(t)$ satisfy the conditions in Theorem \ref{main result 2}, such as
$$\gamma(t)= sgn (t)|t|^\alpha,\ (\alpha\ge2)\quad \textrm{and} \quad \gamma(t)=te^{-1/|t|}.$$
The first one is homogeneous, while another one does not have any
homogeneity. This result is a vector-valued extension of Theorem 3.1
of Nagel and Wainger in \cite{NW76}. Theorem \ref{main result 2}
also generalizes the second author's result \cite{L12e} in the case
$X=\ell^q$ with $5/3<q<5/2$. The proof of Theorem \ref{main result
2} is again motivated by the recent development of the vector-valued
Calder\'on-Zygmund theory \cite{HW08}. In fact, in Section 4, we
prove a more general version, i.e. Theorem \ref{main result 2} is
also true if $X$ satisfies the following weaker condition: there
exist $\theta\in(0,\frac15)$, Hilbert space $H$ and UMD space $Y$
with property $(\alpha)$ (recalled in Section 4) such that
$X=[H,Y]_\theta$.
\section{Proof of Theorem \ref{main result}}
The main arguments in this section are from \cite{SW78},
we will repeat some results for completeness.
Before the proof, we need some notations. Let matrix
$A=diag(\alpha_1,\alpha_2,\cdots,\alpha_n)$, then
$\Gamma'(t)=A\Gamma(t)/t$ for $t>0$. We also define a norm
function $\rho(x)$ by the unique positive solution of
$$\sum^n_{i=1}x^2_i\rho^{-2\alpha_i}=1$$
and $\rho(0)=0$. This definition was introduced in the pioneering
work on anisotropic singular integrals of Fabes \cite{Fab66}.
Obviously, $\rho(\delta_tx)=t\rho(x)$ for $t>0$, $\rho(x)=1$ if and
only if the Euclidean norm $|x|=1$ which means $x$ is on the unit
sphere ${\mathbf S}^{n-1}$. See also Proposition 1-9 in \cite{SW78}
for more properties of $\rho$. By a change of variables, we
assume $\alpha_1=1$ and $\alpha_i\ge1$ for $2\le i\le n$, and set
$\Delta=\alpha_1+\alpha_2+\cdots+\alpha_n$. Without lost of
generality, we assume that $\alpha_i\neq\alpha_j$ when $i\neq j$,
then $\Gamma(t)$ does not lie in a proper subspace of $\mathbb R^n$.
If not, $\Gamma$ lies in some proper subspace, then the argument of
Stein and Wainger in \cite[pp.1262]{SW78} implies our desired
result.
\par
For $z\in\mathbb C$, we define an analytic family of operators
${\mathscr H}_z$ by
$$
\widehat{\mathscr{H}_zf}(\xi)=\{\rho(\xi)\}^zm_z(\xi)\hat{f}(\xi),
$$
where $m_z$ are given by
$$
m_z(\xi)={\rm p.v.}\int_{\mathbb R}e^{-2\pi i\xi
\cdot\Gamma(t)}|t|^z\frac{dt}t.
$$
Obviously, ${\mathscr H}_0$ is our original operator $\mathscr H$.

As in \cite{SW78}, the desired result will be concluded by analytic interpolation once we show the following two estimates: For Hilbert space $H$
\begin{equation}\label{L2H anisotropic}
\big\|{\mathscr H}_zf\big\|_{L^2({\mathbb R}^n;H)}\le
C(z)\big\|f\big\|_{L^2(\mathbb R^n;H)},
\end{equation}
where $-1<Re(z)\le\sigma$ for some $\sigma>0$ and $C(z)$ grows at
most polynomially in $|z|$, and for UMD space $Y$
\begin{equation}\label{LpX0 anisotropic}
\|\mathscr H_zf\|_{L^p(\mathbb R^n;Y)}\le
C(z,p)\|f\|_{L^p(\mathbb R^n;Y)},\ \ 1<p<\infty,
\end{equation}
 where $-\beta\le Re(z)\le-\eta$ for arbitrarily positive $\eta$ and
some positive $\beta$ as well as  $C(z,p)$ grows at most as fast as
a polynomial in $|z|$ for fixed $\eta$.
\par
Indeed, we obtain Theorem
\ref{main result} by performing twice the analytic interpolation
argument in \cite{S56} as follows. Let $T_zf(x)=e^{z^2}{\mathscr H}_zf(x)$. Note that
$|e^{z^2}|=e^{Re(z)^2-Im(z)^2}$, then by \eqref{L2H anisotropic}
there exists a constant $M_0$ which is independent of $Im(z)$ such
that
\begin{equation}\label{2H}
\big\|T_zf\big\|_{L^2(\mathbb R^n;H)}\le C(z)
e^{-Im(z)^2}\big\|f\big\|_{L^2(\mathbb R^n;H)}\le
M_0\big\|f\big\|_{L^2(\mathbb R^n;H)}
\end{equation}
when $-1<{\rm
Re}(z)<\sigma$. Also, for any UMD space $Y$ and $q\in(1,\infty)$, by \eqref{LpX0 anisotropic} there
exists a constant $M_1$ which is independent of $Im(z)$ such that
\begin{equation}\label{qY}
\big\|T_zf\big\|_{L^{q}(\mathbb R^n;{Y})}\le
M_1\big\|f\big\|_{L^{q}(\mathbb R^n;{Y})} \quad when \ \
-\beta<{\rm Re}(z)<0.
\end{equation}
Obviously, this inequality holds in particular with $Y=H$.
\par
 For $1<p<\infty$, we choose $\theta_1\in(0,1)$, $\sigma_1<0$, $0<\sigma_0<\sigma$
 and $q_1\in(1,\infty)$ such that
 \begin{equation}\nonumber
\sigma_0(1-\theta_1)+\sigma_1\theta_1=:\sigma_2>0,\
\frac1p=\frac{1-\theta_1}2+\frac{\theta_1}{q_1}.
 \end{equation}
 Interpolating between \eqref{2H} and \eqref{qY} with $Y=H$, we have
\begin{equation}\label{pH}
\big\|T_zf\big\|_{L^p(\mathbb R^n;H)}\le C(p,z)
\big\|f\big\|_{L^p(\mathbb R^n;H)} \quad when \ \ {\rm
Re}(z)=\sigma_2>0.
\end{equation}
\par
Note that $X=[H,Y]_{\theta}$ for some
Hilbert space $H$, UMD space $Y$ and
$\theta\in(0,1)$. For fixed $\theta$,  we choose $\sigma_3<0$ such
that
\begin{equation}\nonumber
0=(1-\theta){\sigma_2}+\theta\sigma_3.
\end{equation}
In the same way, interpolating between \eqref{pH} and \eqref{qY}
with $q=p$, we obtain
\begin{equation}\nonumber
\|\mathscr Hf\|_{L^p(\mathbb R^n;X)}=\|T_0f\|_{L^p(\mathbb
R^n;X)}\le C\|f\|_{L^p(\mathbb R^n;X)}.
\end{equation}
\par

The estimate \eqref{L2H anisotropic} is trivial since Plancherel's theorem remains true
 for Hilbert space valued functions and the original arguments for Lemma 4.2 in \cite{SW78} work here.
 The novelty of the proof lies in the proof of \eqref{LpX0 anisotropic}.  In the case $Y=\ell^q$ with $1<q<\infty$,
 it has been proved in \eqref{LpX0 anisotropic} in \cite{RRT86} by Benedek, Calder\'on and Panzone's argument since $L^q(\ell^q)$-boundedness is trivial. For general UMD space,
 we shall follow Hyt\"onen and Weis's idea \cite{HW08} established recently to prove the $L^p(Y)$ estimates simultaneously for all $1<p<\infty$. The following subsection is devoted to the proof of estimate \eqref{LpX0 anisotropic}.

\subsection{The proof of \eqref{LpX0 anisotropic}} The following
proof is essentially the same as \cite{H07}, we include it here for the sake of completeness. From Lemma 4.4
of \cite{SW78}, we can write that
$$\mathscr{H}_zf(x)=K_z\ast f(x),$$
where
$$
K_z(x)=\int_{\mathbb R}h_z(x-\Gamma(t))|t|^z\frac{dt}{t}\
\text{and}\ \ \hat{h}_z(\xi)=\{\rho(\xi)\}^z.
$$
It is known that $h_z$ is a locally integrable function,
$C^{\infty}$ away from the origin satisfying
$$h_z(\delta_{\lambda}x)=\lambda^{-\Delta-z}h_z(x),\;\lambda>0,\;x\neq0.$$
 Moreover, each
derivative of $h_z(x)$ is bounded by a polynomial in $|z|$, if
$\rho(x)=1$. In particular, $K_z$ has the homogeneity property
$\lambda ^{\Delta}K_z\big(\delta_\lambda x\big)=K_z(x)$.

Let $\hat{\mathcal{D}}_0(\mathbb{R}^n)=\{\psi\in\mathscr S(\mathbb
R^n)|\ \hat{\psi}\in\mathscr{D}(\mathbb R^n),0\notin supp\
\hat{\psi}\}$. Let $\eta\in\mathcal{D}(\mathbb{R}^n)$ have range
$[0,1]$, vanish for $\rho(\xi)\ge2$ and equal $1$ for
$\rho(\xi)\le1$. For $j\in \mathbb{Z}$, we define
$\hat{\phi}_0(\xi)=\eta(\xi)-\eta(\delta_2\xi)$,
$\hat{\varphi}_j(\xi)=\hat{\phi}_0(\delta_{2^{-j}}\xi)$ and
$\hat{\chi}_j(\xi)=\hat{\phi}_{j-1}(\xi)+\hat{\phi}_{j}(\xi)+\hat{\phi}_{j+1}(\xi)$.
Then $\hat{\varphi}_j(\xi)$ is supported in the annulus
$\{2^{j-1}\leq\rho(\xi)\leq2^{j+1}\}$, and
\begin{align}\label{littlewood paley identity 1}
\sum_j\hat{\varphi_j}(\xi)=1 \ \ \text{for}\ \ \xi\neq0.
\end{align}
Moreover,  since $\hat{\chi}_j$ equals 1 on the support of
$\hat{\phi}_j$, we have
\begin{align}\label{littlewood paley identity 2}
\phi_j=\phi_j\ast \chi_j\ast\chi_j.
\end{align}
The estimate \eqref{LpX0 anisotropic} will be deduced from the
following key estimate which will be shown in the next subsection.

\begin{proposition}\label{pro: key estimate}
Let $\phi_0$ and $K_z$ be defined as above. We have
\begin{align*}
\int_{\mathbb{R}^n}|\phi_0\ast K_z(x)|\log^n(e+\rho(x))dx\leq C(z).
\end{align*}
\end{proposition}

With above preparations at hand, we finish the proof of the estimate
\eqref{LpX0 anisotropic}.
\begin{proof}
For fixed $z$, we denote $K_z$ by $K$ for simplicity. Given
$f\in\hat{\mathcal{D}}_0(\mathbb{R}^n)\otimes Y$,
$g\in\hat{\mathcal{D}}_0(\mathbb{R}^n)\otimes Y^*$, by
\eqref{littlewood paley identity 1} and \eqref{littlewood paley
identity 2}, we have
\begin{align*}
\langle g,K\ast f\rangle&=\langle \tilde{K}\ast
g,f\rangle=\sum_j\langle \phi_j\ast\tilde{K}\ast(\chi_j\ast
g),\chi_j\ast f\rangle,
\end{align*}
where the summation is finite and $\tilde{K}(x)=K(-x)$. Changing
variable and using the fact
$\lambda^{\Delta}K_z(\delta_{\lambda}x)=K_z(x)$,
$$(\phi_j\ast\tilde{K})\ast(\chi_j\ast g)(x)=\int_{\mathbb R^n}\phi_0\ast \tilde{K}(y)(\chi_j\ast g)(x-\delta_{2^{-j}}y)dy.$$
Hence, by H\"older's inequality and the Khintchine-Kahane inequality
\begin{align*}
&\big|\langle g,K\ast f\rangle\big|=\big|\int_{\mathbb R^n}\mathbb{E}\langle \sum_j\epsilon_j\chi_j\ast g(\cdot-\delta_{2^{-j}}y),\sum_i\epsilon_i \phi_0\ast K(y)\chi_i\ast f\rangle dy\big|\\
&\leq \int_{\mathbb R^n}\mathbb{E}\| \sum_j\epsilon_j\chi_j\ast
g(\cdot-\delta_{2^{-j}}y)\|_{L^{p'}(Y^*)}
\mathbb{E}\|\sum_i\epsilon_i\chi_i\ast f\|_{L^p(Y)}|\phi_0\ast K(y)|
dy.
\end{align*}
It is easy to check that $m=\sum_j\epsilon_j\hat{\chi}_j$ is an
anisotropic multiplier. Hence, by Theorem 3 in \cite{H07}, we have
\begin{align}\label{anisotropic multiplier}
\|\sum_{j}\epsilon_j\chi_j\ast f\|_{L^p(\mathbb R^n;Y)}\leq
C_{p,X}\|f\|_{L^p(\mathbb R^n;Y)}.
\end{align}
By Proposition {\ref{pro: key estimate}} and
\eqref{anisotropic multiplier}, we shall finish the proof by showing
\begin{align*}
\mathbb{E}\|\sum_j\epsilon_j\chi_j\ast
g(\cdot-\delta_{2^{-j}}y)\|_{L^{p'}(Y^*)}\leq
C\log^n(e+\rho(y))\mathbb{E} \|\sum_j\epsilon_j\chi_j\ast
g\|_{L^{p'}(Y^*)}.
\end{align*}
Let $e_i$ be the $i$-th standard unit vector. Above estimate is just a $n$-fold application of
\begin{align*}
\mathbb{E}\|\sum_j\epsilon_j\chi_j\ast
g(\cdot-\delta_{2^{-j}}y_ie_i)\|_{L^{p'}(Y^*)}\leq
C\log(e+\rho(y))\mathbb{E} \|\sum_j\epsilon_j\chi_j\ast
g\|_{L^{p'}(Y^*)},
\end{align*}
which follows from
Lemma 10 of  Bourgain \cite{Bou86}.
\end{proof}

\subsection{The proof of Proposition \ref{pro: key estimate}}
The proof  of Proposition \ref{pro: key estimate} is based on the
following two lemmas. The first one states that the kernel $K_z$
satisfies a weighted H\"ormander condition, which will be verified
at the end of this subsection.
\begin{lemma}\label{lem: weighted homander condition}
If $-\beta\leq Re(z)\leq-\eta$, then for sufficiently large constants $C_0$ and $C_1(z)$, we have
\begin{align}\label{weighted homander condition}
\int_{\rho(x)\geq C_0\rho(y)}|K_z(x-y)-K_z(x)|\log^n(e+\rho(x))dx\leq C_1(z)\log^n(e+\rho(y))
\end{align}
for any $y\in\mathbb{R}^n\setminus\{0\}$.
Moreover, $C_1(z)$ grows at most as fast as a polynomial in $|z|$ for a fixed $\eta$.
\end{lemma}
The second lemma is a kind of decomposition lemma which has been
established in Lemma 4.10 of \cite{HW08}. We reformulate it in our
anisotropic case.
\begin{lemma}\label{lem: decomposition lemma}
Let $\varphi\in\mathscr{S}(\mathbb{R}^n)$ with vanishing integral.
Then there exists a decomposition $\varphi=\sum_{m\geq0}\psi_m$ with
the following properties:
\begin{align*}
\psi_m\in\mathcal{D}(\mathbb{R}^n),\;\mathrm{supp}\psi_m\subseteq
\{x|\ \rho(x)\le C2^{\alpha m}\},\;\int_{\mathbb R^n}\psi_m(y)dy=0,
\end{align*}
where $C$ and $\alpha$ are two universal constants only depending on
the norm $\rho$ and the dimension $n$, and for every
$p\in[1,\infty]$ and every $M>0$, the sequence of Lebesgue norms
$\|\psi_m\|_{L^p}$, as well as $\|\hat{\psi}_m\|_{L^p}$, is
$\mathcal{O}(2^{-mM})$ as $m\rightarrow\infty$.
\end{lemma}
\begin{proof}Let us give a quick explanation of this lemma.
From Lemma 4.10 of \cite{HW06}, $\psi_m$ is supported in $\{x|\
|x|\le 2^m\}$. Fix $x\in\{x|\ |x|\le 2^m\}$, by Proposition 1-9 of
\cite{SW78}, if $\rho(x)\geq1$, then
$$\rho(x)\leq c_1|x|^{\alpha_1}\leq c_12^{a_1m}$$
and if $\rho(x)\leq1$, then
$$\rho(x)\leq c_2|x|^{a_2}\leq c_22^{a_2m}$$
with $c_1,c_2,a_1,a_2$ positive constants. We obtain the desired
result by choosing $C=\max\{c_1,c_2\}$ and $\alpha=\max\{a_1,a_2\}$.
\end{proof}
\begin{proof}[Proof of Proposition {\ref{pro: key estimate}}] The main idea comes from \cite{HW06}, we include most details here for completeness.
By Lemma \ref{lem: decomposition lemma}, we write
$\phi_0=\sum_{m\geq0}\psi_m$ with $\psi_m$'s satisfying the
properties stated in that lemma. Then we decompose $K_z$ into pieces
$$K_{z,m}(x)=K_z\ast\psi_m(x)$$
and estimate each of them respectively.

We first estimate the integral outside the larger
ellipsoid $\mathcal B_1=\{x|\ \rho(x)\le CC_12^{\alpha m}\}$ with
$C_1$ fixed later depending on $C_0$. Recall that $\psi_m$ is
supported in the ellipsoid $\mathcal B_0=\{x|\ \rho(x)\le C2^{\alpha
m}\}$ and the integral of $\psi_m$ vanishes, by Fubini's theorem  and
Lemma \ref{lem: weighted homander condition}, we obtain
\begin{align*}
&\int_{\mathcal B_1^c}|K_{z,m}(x)|\log^n(e+\rho(x))dx\\
&=\int_{\mathcal B_1^c}|\int_{\mathcal
B_0}K_{z}(x-y)\psi_m(y)dy|\log^n(e+\rho(x))dx\\
&\leq \int_{\mathcal B_0}\int_{\rho(x)\geq
C_0\rho(y)}|K_{z}(x-y)-K_{z}(x)|\log^n(e+\rho(x))dx\psi_m(y)dy\\
&\leq C_1(z) \int_{\mathcal B_0}\log^n(e+\rho(y))\psi_m(y)dy\leq
C_1(z)\|\psi_m\|_{L^{\infty}} \int_{\mathcal
B_0}\log^n(e+\rho(y))dy.
\end{align*}
By Lemma \ref{lem: decomposition lemma}, the last quantity is of
order $\mathcal{O}(2^{-m})$ as $m\rightarrow\infty$ since
$\|\psi_m\|_{L^{\infty}}\leq C_M2^{-mM}$ for $M>0$ while
$$\int_{\mathcal B_0}\log^n(e+\rho(y))dy\leq C2^{mN}$$
for a fixed $N$.

Inside the ellipsoid $\mathcal B_1$, the computation is easier because of the fact $\|\hat{K}_z\|_{L^{\infty}}\leq C(z)$, then
\begin{align*}
\int_{\mathcal B_1}|K_{z,m}(x)|&\log^n(e+\rho(x))dx\leq \|K_{z,m}\|_{L^{\infty}}\int_{\mathcal B_1}\log^n(e+\rho(x))dx\\
&\leq \int_{\mathcal B_1}\log^n(e+\rho(x))dx\|\hat{K}_{z,m}\|_{L^1}\\
&=\int_{\mathcal B_1}\log^n(e+\rho(x))dx\int_{\mathbb{R}^n}|\hat{K}_z(\xi)\hat{\psi}_m(\xi)|d\xi\\
&\leq\|\hat{K}_z\|_{L^{\infty}}\|\hat{\psi}_m\|_{L^1}\int_{\mathcal
B_1}\log^n(e+\rho(x))dx \leq C(z)2^{-m}.
\end{align*}
The last inequality holds due to the same reason that for the case
outside the ellipsoid. Finally, we obtain Proposition \ref{pro: key
estimate} by summing over $m$.
\end{proof}
To complete the proof of Proposition \ref{pro: key
estimate}, we still need to show Lemma \ref{lem: weighted homander
condition}.
\begin{proof}[Proof of Lemma \ref{lem: weighted homander condition}]
We follow the main sketch provided in \cite{SW78}, but improve related estimates. To verify $K_z$ satisfying \eqref{weighted homander condition}, we
may assume that $\rho(y)=1$, it suffices to prove that
\begin{equation}\label{horcon}
\int_{\rho(x)\ge C_0}|K_z(x-y)-K_z(x)|\log^n\big(e+\rho(x)\big)dx\le
C(z).
\end{equation}
\par
In fact, we set $\lambda=\rho(y)$ and $y'=y/\lambda$. Obviously,
$\rho(y')=1$. By a linear transformation $x=\delta_\lambda x'$ and the homogeneity of $K_z$, we have
\begin{align*}
&\int_{\rho(x)\ge
C_0\rho(y)}|K_z(x-y)-K_z(x)|\log^n\big(e+\rho(x)\big)dx\\
=&
\int_{\rho(x')\ge
C_0}|K_z(x'-y')-K_z(x')|\log^n\big(e+\lambda\rho(x')\big)dx'.
\end{align*}
If $\lambda=\rho(y)\ge6$, it is trivial that
\begin{equation}\nonumber
\log\big(e+\lambda\rho(x')\big)\le\log\big(e+\lambda\big)+\log\big(e+\rho(x')\big)\le
\log\big(e+\lambda\big)\log\big(e+\rho(x')\big),
\end{equation}
where we use the assumption that $C_0\ge6$. Then,
\begin{align*}
&\int_{\rho(x)\ge
C_0\rho(y)}|K_z(x-y)-K_z(x)|\log^n\big(e+\rho(x)\big)dx\\
&\le \int_{\rho(x')\ge
C_0}|K_z(x'-y')-K_z(x')|\log^n\big(e+\rho(x')\big)dx'\log^n\big(e+\rho(y)\big)\\
&\le C(z)\log^n\big(e+\rho(y)\big).
\end{align*}
When $\lambda=\rho(y)<6$, by \eqref{horcon}, we get
\begin{align*}
&\int_{\rho(x)\ge
C_0\rho(y)}|K_z(x-y)-K_z(x)|\log^n\big(e+\rho(x)\big)dx\\
&\le 2^n\int_{\rho(x')\ge
C_0}|K_z(x'-y')-K_z(x')|\log^n\big(e+\rho(x')\big)dx'\\
&\le C(z)\le C(z)\log^n\big(e+\rho(y)\big).
\end{align*}
\par
To prove \eqref{horcon}, we define $K_z^1$ and $K_z^2$ by
$$
K_z^1(x)=\int_{|t|\le1}h_z(x-\Gamma(t))|t|^z\frac{dt}t\ \text{and}\
K_z^2(x)=K_z(x)-K_z^1(x),
$$
respectively. We split the integral as
\begin{align*}
&\int_{\rho(x)\ge C_0}|K_z(x-y)-K_z(x)|\log^n\big(e+\rho(x)\big)dx\\
&\le \int_{\rho(x)\ge C_0}|K_z^1(x)|\log^n\big(e+\rho(x)\big)dx\\
& +\int_{\rho(x)\ge C_0}|K_z^1(x-y)|\log^n\big(e+\rho(x)\big)dx\\
&+\int_{\rho(x)\ge
C_0}|K_z^2(x-y)-K_z^2(x)|\log^n\big(e+\rho(x)\big)dx.
\end{align*}

To estimate first two summands, we need a estimate related to $h_z$,
which can be found in \cite[pp.1273]{SW78}. The homogeneity
and smoothness of $h_z$ away from origin imply that
\begin{equation}\label{diff}
|h_z(x-y)-h_z(x)|\le C(z)\frac{|y|}{\{\rho(x)\}^{\Delta+Re(z)+\mu}}
\end{equation}
for some $\mu>0$, provide $|y|/|x|$ is sufficiently small.

We set $\beta=\min\{\mu,1\}$. For the first integral, by using
Fubini's theorem and \eqref{diff},
 we have
\begin{align*}
&\int_{\rho(x)\ge C_0}|K_z^1(x)|\log^n\big(e+\rho(x)\big)dx\\
&\le\int_{\rho(x)\ge C_0}\int_{|t|\le1}|h_z(x-\Gamma(t))-h_z(x)||t|^{Re(z)-1}dt\log^n\big(e+\rho(x)\big)dx\\
&\le\int_{|t|\le1}|t|^{Re(z)-1}\int_{\rho(x)\ge C_0}|h_z(x-\Gamma(t))-h_z(x)|\log^n\big(e+\rho(x)\big)dxdt\\
&\le\int_{|t|\le1}|t|^{Re(z)-1}|\Gamma(t)|\int_{\rho(x)\ge C_0}\rho(x)^{-[\Delta+Re(z)+\mu]}\log^n\big(e+\rho(x)\big)dxdt\\
&\le C(z),
\end{align*}
where we use the fact that $-\beta<Re(z)<0$.
\par
The norm function $\rho(x)$ have the property of $\rho(x+y)\le
c\big(\rho(x)+\rho(y)\big)$ for some $c>0$(see Proposition 1-9 in
\cite{SW78}). Specially, we set $C_0\ge\max\{6,3c\}$. Note that
$\rho(x-y)\ge \frac1c\rho(x)-\rho(y)\ge \frac{C_0}c-1\ge2$ and
$\rho(x)\le c[\rho(x-y)+\rho(y)]\le c\rho(x-y)+c$. Using a linear
transformation, we treat the second summand as the first one,
\begin{align*}
&\int_{\rho(x)\ge C_0}|K_z^1(x-y)|\log^n\big(e+\rho(x)\big)dx\\
&\le\int_{\rho(x)\ge 2}|K_z^1(x)|\log^n\big(e+c+c\rho(x)\big)dx\le C(z).
\end{align*}

Finally, using Fubini's theorem, we have
\begin{eqnarray*}
&&\int_{\rho(x)\ge C_0}|K_z^2(x-y)-K_z^2(x)|\log^n\big(e+\rho(x)\big)dx\\
&\le&\int_{|t|\ge1}\int_{\rho(x)\ge
C_0}\big|h_z\big(x-y-\Gamma(t)\big)-h_z\big(x-\Gamma(t)\big)\big|\log^n\big(e+\rho(x)\big)\frac{dxdt}{|t|^{1-Re(z)}}.
\end{eqnarray*}
We divide the inner integral above according to the distance between
$x$ and $\Gamma(t)$. Note that $\rho(y)=1$, if $|y|/|x-\Gamma(t)|$
is sufficient small, that is $|x-\Gamma(t)|$ is away from the
origin, we can get that $\rho(x-\Gamma(t))\ge C_2$, where $C_2$ is an
appropriate constant. In this case, by \eqref{diff} and a
linear transformation, we obtain the following estimate
\begin{align*}
&\int_{|t|\ge1}\int_{\substack {\rho(x)\ge C_0\\ \rho(x-\Gamma(t))\ge C_2}}\big|h_z\big(x-y-\Gamma(t)\big)-h_z\big(x-\Gamma(t)\big)\big|\log^n\big(e+\rho(x)\big) \frac{dxdt}{|t|^{1-Re(z)}}\\
&\le C\int_{|t|\ge1}\int_{\substack {\rho(x)\ge C_0\\ \rho(x-\Gamma(t))\ge C_2}}\frac{|y|}{\{\rho\big(x-\Gamma(t)\big)\}^{\Delta+\mu+Re(z)}}\log^n\big(e+\rho(x)\big)\frac{dxdt}{|t|^{1-Re(z)}}\\
&\le C\int_{|t|\ge1}\int_{\rho(x)\ge C_2}\frac{1}{\{\rho(x)\}^{\Delta+\mu+Re(z)}}\log^n\big(e+c\rho(x)+ct\big)\frac{dxdt}{|t|^{1-Re(z)}}\\
&\le C\int_{|t|\ge1}\int_{\rho(x)\ge C_2}\frac{1}{\{\rho(x)\}^{\Delta+\mu+Re(z)}}\big\{\log^n\big(e+\rho(x)\big)+\log^n\big(e+t\big)\big\}\frac{dxdt}{|t|^{1-Re(z)}}\\
&\le C,
\end{align*}
where we use the fact that for fixed $|t|\ge1$, $\rho(x)\le
c[\rho(x-\Gamma(t))+\rho(\Gamma(t))]=c[\rho(x-\Gamma(t))+t]$.
\par
It is trivial that $\rho\big(x+y+\Gamma(t)\big)\le
c^2[\rho(x)+\rho(y)+\rho(\Gamma(t))]=c^2[1+\rho(x)+t]$. Then, the
remainder can be controlled by
 \begin{align*}
&\int_{|t|\ge1}\int_{\substack {\rho(x)\ge C_0\\ \rho(x-\Gamma(t))\le C_2}}[|h_z(x-y-\Gamma(t))|+|h_z(x-\Gamma(t))|]\log^n(e+\rho(x))\frac{dxdt}{|t|^{1-Re(z)}}\\
&\le\int_{|t|\ge1}\int_{\substack {\rho(x)\ge C_0\\ \rho(x-\Gamma(t))\le C_2}}|h_z\big(x-y-\Gamma(t)\big)|\log^n\big(e+\rho(x)\big)dx |t|^{Re(z)-1}dt\\
&+\int_{|t|\ge1}\int_{\substack {\rho(x)\ge C_0\\ \rho(x-\Gamma(t))\le C_2}}|h_z\big(x-\Gamma(t)\big)|\log^n\big(e+\rho(x)\big)dx |t|^{Re(z)-1}dt\\
\end{align*}
\begin{align*}
&\le C\int_{|t|\ge1}\int_{\substack {\rho(x)\le c(C_2+1)}}|h_z(x)|dx |t|^{Re(z)-1}\log^n(e+t)dt\\
&\le C(z),
\end{align*}
where we use the fact that $h_z$ is locally integrable.
\end{proof}

\section{Anisotropic singular integrals}
It was shown by Calder\'on and Zygmund \cite{CaZy52} that the
$L^p$-boundedness of singular integrals with rough kernels can be
deduced from the $L^p$-boundedness of the (directional) Hilbert
transform using the method of rotations. In this section, we show a
similar phenomenon happens, that is, the $L^p(X)$-boundedness of
Hilbert transforms along curve $\Gamma(t)=(|t|^{\alpha_1}sgn
t,|t|^{\alpha_2}sgn t,\cdots,|t|^{\alpha_n}sgn t)$ considered in the
previous section implies the $L^p(X)$ boundedness of singular
integrals $T_{\Omega}$ with kernels of the form
$K(x)=\Omega(x)\rho(x)^{-\Delta}$, where $\Omega$ is a function on $\mathbb{R}^n\setminus\{0\}$ satisfying
the homogeneity $\Omega(\delta_tx)=\Omega(x)\ \text{for all}\ t>0,$
size condition
\begin{equation}\label{siz}
\int_{\mathbf
S^{n-1}}\sum^n_{i=1}\alpha_i\omega^2_i|\Omega(\omega)|d\omega<\infty,
\end{equation}
 and the cancelation condition
\begin{equation}\nonumber
\int_{\mathbf
S^{n-1}}\sum^n_{i=1}\alpha_i\omega^2_i\Omega(\omega)d\omega=0,
\end{equation}
which can be understood from the following change-of-variable formula
$$dx=t^{\Delta-1}\sum^n_{i=1}\alpha_i{\omega}^2_idtd\omega.$$

\begin{theorem}\label{app}
Let $X\in\mathcal I$. If $\Omega$ is odd,  then the
operators $T_{\Omega}$ described previously  are bounded on $L^p(\mathbb
R^n;X)$ for $1<p<\infty$.
\end{theorem}

Guliev \cite{G93} has obtained the boundedness of anisotropic
singular integrals with scalar valued-kernels on UMD lattices.
Recently, Hyt\"{o}nen\cite{H07} generalized some work of Guliev to
the anisotropic singular integrals with operator-valued kernels
acting on UMD space. While their arguments require that $\Omega(x)$
should satisfy a kind of $L^\infty$-Dini condition, which is a much
more restricted condition than ours. So, Theorem \ref{app} is a
generalization of Hyt\"{o}nen and Guliev's result in this sense.

\begin{proof}
Changing the variables, we find
\begin{align}\nonumber
T_{\Omega}f(x)&={\rm p.v.}\int_{\mathbb
R^{n}}f\big(x-\delta_{\rho(y)}\delta_{\rho(y)}^{-1}y\big)\Omega(\delta_{\rho(y)}^{-1}y)\{\rho(y)\}^{-\Delta}dy\\
\label{p} &=\int_0^\infty\int_{ \mathbf
S^{n-1}}f\big(x-\delta_t\omega\big)\sum^n_{i=1}\alpha_i\omega^2_i\Omega(\omega)d\omega\frac{dt}t.
\end{align}
Note that $\Omega$ is odd, by a linear transformation, we also have
\begin{equation}\label{n}
T_{\Omega}f(x)=\int^0_{-\infty}\int_{ \mathbf
S^{n-1}}f\big(x+\delta_{(-t)}\omega\big)\sum^n_{i=1}\alpha_i\omega^2_i\Omega(\omega)d\omega\frac{dt}t.
\end{equation}
Using Fubini theorem, and adding \eqref{p} and \eqref{n} together,
we get
$$
T_{\Omega}f(x)=\frac12\int_{ \mathbf
S^{n-1}}\sum^n_{i=1}\alpha_i\omega^2_i\Omega(\omega)\big[\int_{-\infty}^0f\big(x+\delta_{(-t)}\omega\big)\frac{dt}t+\int_0^\infty
f\big(x-\delta_t\omega\big)\frac{dt}t\big]d\omega.
$$
Then, it suffices to prove that
$$
\|\int_{-\infty}^0f\big(x+\delta_{(-t)}\omega\big)\frac{dt}t+\int_0^\infty
f\big(x-\delta_t\omega\big)\frac{dt}t\|_{L^p(\mathbb R^n;\mathbf
X)}\le C_p\|f\|_{L^p(\mathbb R^n;X)},
$$
where the constant $C_p$ is independent of $\omega$.

For fixed $\omega\in \mathbf S^{n-1}$, define $\Gamma_{\omega}(t)$
as the curve in the form of \eqref{ts} associated to the dilation
$\delta_t$ with $\mathbf e=\omega$ and  $\mathbf f=-\omega$, then
the quantity inside the norm of the previous inequality is the
Hilbert transform along the curve $\Gamma_\omega(t)$. The same
arguments for the proof of Theorem \ref{main result} work also for
the curve $\Gamma_\omega(t)$,  and we obtain the desired result.
\end{proof}

In the classical case (dilation given by $\delta_tx=tx$), it is
known that the boundedness of $T_{\Omega}$ is also obtained for the
even function $\Omega$ under a stronger size condition $\Omega\in
L\log ^+L(\mathbf S^{n-1})$. The main ingredient is the existence of Riesz transforms $R_j,\;j=1,2,\cdots,n$, such that
\begin{enumerate}[(i)]
\item $-\sum^n_{j=1}R_j\circ R_j=I$,
\item the kernel of $T_{\Omega}\circ R_j$ is still homogeneous, and the associated $\Omega_j$ is an odd
function satisfying size condition \eqref{siz}.
\end{enumerate}

In the anisotropic setting, it seems very difficult to find some
replacements for Riesz transforms such that similar properties
as (i) and (ii) hold. Hence we leave it as an open problem that
whether Theorem \ref{app} is still true for the  even function
$\Omega$ under a stronger size condition.

\section{The proof of Theorem \ref{main result 2}}
The main argument for the proof is similar to that for Theorem \ref{main
result}. We first introduce a family of analytic operators. For
$z\in\mathbb C$, we define an analytic family of operators
${\mathscr H}_z$ by
$$
\widehat{\mathscr{H}_zf}(\xi,\eta)=m_z(\xi,\eta)\hat{f}(\xi,\eta),
$$
where $m_z$ are given by
$$
m_z(\xi,\eta)={\rm p.v.}\int_{\mathbb R}e^{-2\pi i[\xi
t+\eta\gamma(t)]}\big[1+\eta^2\gamma^2(t)\big]^z\frac{dt}t.
$$
Obviously, ${\mathscr H}_0$ is our original operator $\mathscr H$.

Following the idea in \cite{NW76}, it suffices to prove the following two estimates:
\begin{equation}\label{L2H}
\big\|{\mathscr H}_zf\big\|_{L^2({\mathbb R}^2;H)}\le
C_\delta\big[1+|Im(z)|\big]\big\|f\big\|_{L^2(\mathbb R^2;
H)},
\end{equation}
where $Re(z)=\frac14-\delta$ for some $\delta>0$, and
\begin{equation}\label{Lplq}
\big\|\mathscr H_z f\big\|_{L^q({\mathbb R}^2;{Y})}\le
C\big[1+|Im(z)|\big]^2\big\|f\big\|_{L^q({\mathbb R}^2;{
Y})},
\end{equation}
where $Y$ is an UMD lattice, $Re(z)<-1$, $1<q<\infty$, the
constant $C$ depends on $Re(z)$ and is independent of $Im(z)$.

Indeed, we finish the proof by analytic
interpolation argument \cite{S56}.  Let $T_zf(x)=e^{z^2}{\mathscr
H}_zf(x)$. Note that $|e^{z^2}|=e^{Re(z)^2-Im(z)^2}$, by \eqref{L2H}
there exists a constant $M_0$ which is independent of $Im(z)$ such
that
$$
\big\|T_zf\big\|_{L^2(\mathbb R^2;H)}\le C_\delta
e^{-Im(z)^2}\big[1+|Im(z)|\big]\big\|f\big\|_{L^2(\mathbb
R^2;H)}\le M_0\big\|f\big\|_{L^2(\mathbb R^2;H)}
$$
when ${\rm Re}(z)=\frac14-\delta$. Also, for UMD lattice $Y$ and $q\in(1,\infty)$, by \eqref{Lplq}
there exists a constant $M_1$ which is independent of $Im(z)$ such
that
$$
\big\|T_zf\big\|_{L^{q}(\mathbb R^2;{Y})}\le
M_1\big\|f\big\|_{L^{q}(\mathbb R^2;{Y})} \quad when \ \
{\rm Re}(z)<-1.
$$
This inequality also holds in particular with $Y=H$.

For $\frac53<p\le2$, there exist $1<q<\infty$ and
$\theta_0\in(0,\frac15)$  so that
$$
\frac1p=\frac{1-\theta_0}2+\frac{\theta_0}{q}\ \ \text{and}\ \
(\frac14-\delta)(1-\theta_0)+(-1-\varepsilon_0)\theta_0=:\sigma_1\in(0,\frac14)
$$
for some $\varepsilon_0>0$ and $0<\delta<\frac14$. By interpolation
of analytic operators, we have
$$
\big\|T_zf\big\|_{L^p(\mathbb R^2;H)}\le
C(z)\big\|f\big\|_{L^p(\mathbb R^2;H)} \quad for \ \ {\rm
Re}(z)=\sigma_1\in(0,1/4).
$$

Given an UMD lattice $X\in\mathcal{I}_{(0,1/5)}$, there exist
a $\theta\in(0,\frac15)$, a Hilbert space $H$ and another UMD
lattice $Y$, such that $L^p(\mathbb R^2;X)=[L^p(\mathbb R^2;H),L^{p}(\mathbb R^2;{Y})]_{\theta}$. For such a $\theta$ and appropriate $\sigma_1$, we choose
$\varepsilon_1>0$ such that
$(1-\theta)\sigma_1+\theta(-1-\varepsilon_1)=0$.
 Using interpolation of
analytic operators once more, we obtain
\begin{equation}\nonumber
\big\|\mathscr H f\big\|_{L^p(\mathbb R^2;X)}\le
C\big\|f\big\|_{L^p(\mathbb R^2;X)}
\end{equation}
for $\frac53<p\le 2$. The duality argument implies the result for $2\le
p<\frac52$. This completes the proof of Theorem \ref{main result 2}.

\par

The estimate \eqref{L2H} holds since Plancherel's theorem works also for Hilbert space valued functions and the original argument in \cite{NW76} can be repeated in the present situation. The novelty of the proof lies in the estimate \eqref{Lplq}, for which we need the vector-valued Fourier multiplier
theorem established recently.

 Let us firstly recall some notations. A Banach space $X$ satisfies property $(\alpha)$ if there is
 a positive
constant $C$ such that
\begin{equation}\nonumber
\mathbb{E}\mathbb{E}'\bigg|\sum^N_{k,l=1}\epsilon_k\epsilon_{l}'\alpha_{kl}x_{kl}\bigg|_{X}\le
C\mathbb{E}\mathbb{E}'\bigg|\sum^N_{k,l=1}\epsilon_k\epsilon_{l}'x_{kl}\bigg|_{X}
\end{equation}
for all $N\in\mathbb N$, all vectors $x_{kl}\in X$ and scalars
$|\alpha_{kl}|\le 1$ $(1\le k,l\le N)$, where $\epsilon_k$,
$k\in\mathbb Z$ and $\epsilon_l'$, $l\in\mathbb Z$ are two
identical independent sequences.
\begin{remark}\label{lust} The commutative $L^p$ spaces satisfy property $(\alpha)$ for all
$1\le p<\infty$. Also, this property is inherited from $X$
by $L^p(\mu,X)$ for $p\in[1,\infty)$. Every Banach space
with a local unconditional structure and finite cotype, in
particular every Banach lattice, has property $(\alpha)$.
\end{remark}
Let $m:{\mathbb R}^n\rightarrow \mathbb C$ be a bounded function, the
associated operator $T_m$ is defined on the test functions
$f\in{\mathscr S}({\mathbb R}^n)\otimes X$ by
$$
T_mf(x)=(m\hat{f})^{\vee}(x).
$$
The sufficiency part of  the following vector-valued Fourier multiplier theorem was
proved by \v{S}trkalj and Weis \cite{SW07}, while the necessity of those
conditions was obtained by Hyt\"{o}nen and Weis \cite{HW08}.
\begin{lemma}\label{multiplier theorem}
The Marcinkiewicz-Lizorkin condition $|\xi^\beta||D^\beta m(\xi)|\le
C$ for all $\beta\in\{0,1\}^n$ is sufficient for the $L^p({\mathbb R}^n;X)$-boundedness of $T_m$, $n>1$, if and only if
$X$ is an UMD space with property $(\alpha)$.
\end{lemma}
In view of Lemma \ref{multiplier theorem} and Remark \ref{lust}, to prove the estimate  \eqref{Lplq}, it suffices to show that the following functions
$$
m_z(\xi,\eta),\ \xi\frac{\partial m_z}{\partial\xi}(\xi,\eta),\
\eta\frac{\partial m_z}{\partial\eta}(\xi,\eta),\
\xi\eta\frac{\partial^2 m_z}{\partial\xi\partial\eta}(\xi,\eta)
$$
are uniformly bounded on $\mathbb R^2$ for $Re(z)<-1$.
\par
The uniform boundedness of $m_z(\xi,\eta)$ is trivial, it can be
showed by minor modification of the proof of \eqref{L2H}.
Without repetition, we omit the proof. The following
estimates are essentially proved in \cite{NW76}, we include them here for the sake of completeness.

{\bf The boundedness of $\xi\frac{\partial
m_z}{\partial\xi}(\xi,\eta)$.}
Integration by part implies that
\begin{eqnarray*}
\xi\frac{\partial m_z}{\partial\xi}(\xi,\eta)&=&-2\pi i\int_{\mathbb
R}e^{-2\pi i[\xi
t+\eta\gamma(t)]}\xi\big[1+\eta^2\gamma^2(t)\big]^zdt\\
&=&\int_{\mathbb R}\frac{d}{dt}(e^{-2\pi i\xi
t})e^{-2\pi i\eta\gamma(t)}\big[1+\eta^2\gamma^2(t)\big]^zdt\\
&=&e^{-2\pi i[\xi
t+\eta\gamma(t)]}\big[1+\eta^2\gamma^2(t)\big]^z\bigg|^{\infty}_{-\infty}\\
&+&2\pi
i\eta\int_{\mathbb R}e^{-2\pi i[\xi
t+\eta\gamma(t)]}\gamma'(t)\big[1+\eta^2\gamma^2(t)\big]^zdt\\
&-&2z\eta^2\int_{\mathbb R}e^{-2\pi i[\xi
t+\eta\gamma(t)]}\big[1+\eta^2\gamma^2(t)\big]^{z-1}\gamma(t)\gamma'(t)dt.
\end{eqnarray*}
\par
Note that $Re(z)<-1$, for $t\in{\mathbb R}$, we have $
\big|\big[1+\eta^2\gamma^2(t)\big]^z\big|=\big[1+\eta^2\gamma^2(t)\big]^{Re(z)}\le1$.
The boundary terms are bounded by $1$.
\par
For $Re(z)<-1$, making the change of variables $u=|\eta|\gamma(t)$,
we obtain
\begin{eqnarray*}
\bigg|\eta\int_{\mathbb R}e^{-2\pi i[\xi
t+\eta\gamma(t)]}\gamma'(t)\big[1+\eta^2\gamma^2(t)\big]^zdt\bigg|&\le&\int_{\mathbb
R}\gamma'(t)|\eta|\big[1+\eta^2\gamma^2(t)\big]^{Re(z)}dt\\
&\le&\int_{\mathbb R}\big(1+u^2\big)^{Re(z)}du\le\pi.
\end{eqnarray*}
\par
In a similar way, the second integrated term can be dominated by
\begin{eqnarray*}
&&\bigg|z\eta^2\int_{\mathbb R}e^{-2\pi i[\xi
t+\eta\gamma(t)]}\big[1+\eta^2\gamma^2(t)\big]^{z-1}\gamma(t)\gamma'(t)dt\bigg|\\
&\le&2|z|\int_{0}^\infty\big[1+\eta^2\gamma^2(t)\big]^{Re(z)-1}\eta^2\gamma(t)\gamma'(t)dt\\
&\le&|z|\int_{0}^\infty(1+u)^{Re(z)-1}du\le1+|Im(z)|.
\end{eqnarray*}
Therefore, for $Re(z)<-1$,
$$
\big|\xi\frac{\partial m_z}{\partial\xi}(\xi,\eta)\big|\le
C\big[1+|Im(z)|\big].
$$

{\bf The boundedness of $\eta\frac{\partial m_z}{\partial\eta}(\xi,\eta)$.}
Integrating by parts, we obtain
\begin{eqnarray*}
\eta\frac{\partial m_z}{\partial\eta}(\xi,\eta)&=&-2\pi
i\ {\rm p.v.}\int_{\mathbb R}e^{-2\pi i[\xi
t+\eta\gamma(t)]}\eta\gamma(t)\big[1+\eta^2\gamma^2(t)\big]^z\frac{dt}t\\
&+&2z\ {\rm p.v.}\int_{\mathbb R}e^{-2\pi i[\xi
t+\eta\gamma(t)]}\eta^2\gamma^2(t)\big[1+\eta^2\gamma^2(t)\big]^{z-1}\frac{dt}t.
\end{eqnarray*}
\par
To estimate above two integrals, we follow the argument used in the proof of \eqref{L2H}. For the first integral, for any
$\varepsilon>0$, it suffices to bound the following two parts
$$
\int_{\varepsilon<|t|<
t_0}|\eta||\gamma(t)|\big[1+\eta^2\gamma^2(t)\big]^{Re(z)}\frac{dt}{|t|}\
\ \text{and}\ \ \int_{|t|\ge
t_0}|\eta||\gamma(t)|\big[1+\eta^2\gamma^2(t)\big]^{Re(z)}\frac{dt}{|t|}.
$$
Recall that $t_0>0$ was chosen so that $|\eta|\gamma(t_0)=1$, and
$\gamma(t)\le t\gamma'(t)$ because of the convexity. Thus,
\begin{eqnarray*}
\int_{\varepsilon<|t|<
t_0}|\eta||\gamma(t)|\big[1+\eta^2\gamma^2(t)\big]^{Re(z)}\frac{dt}{|t|}\le
2|\eta|\int_0^{t_0}\frac{\gamma(t)}{t}dt\le2|\eta|\int_0^{t_0}\gamma'(t)dt\le
2.
\end{eqnarray*}
For $Re(z)<-1$, an elementary calculation implies that
$$
\int_{|t|\ge
t_0}|\eta||\gamma(t)|\big[1+\eta^2\gamma^2(t)\big]^{Re(z)}\frac{dt}{|t|}\le2|\eta|^{2Re(z)+1}\int_{
t_0}^\infty\gamma^{2Re(z)}(t)\frac{\gamma(t)}{t}dt\le2.
$$
\par
Similarly, the second integral can be controlled by
\begin{eqnarray*}
&&\bigg|z\int_{\mathbb R}e^{-2\pi i[\xi
t+\eta\gamma(t)]}\eta^2\gamma^2(t)\big[1+\eta^2\gamma^2(t)\big]^{z-1}\frac{dt}t\bigg|\\
&\le& 2|z|\int_{0}^{t_0}\eta^2\gamma^2(t)\frac{dt}{t}
+2|z|\int_{t_0}^\infty\eta^2\gamma^2(t)\big[\eta^2\gamma^2(t)\big]^{Re(z)-1}\frac{dt}{t}\\
&\le&2|z|\eta^2\int_{0}^{t_0}\gamma(t)\gamma'(t)dt+2|z|\eta^{2Re(z)}\int_{t_0}^\infty\gamma^{2Re(z)-1}(t)\gamma'(t)dt\\
&\le&|z|+\frac{|z|}{|Re(z)|}\le 2|Re(z)|\big[1+|Im(z)|\big].
\end{eqnarray*}
Therefore, for $Re(z)<-1$,
$$
\big|\xi\frac{\partial m_z}{\partial\xi}(\xi,\eta)\big|\le
C\big[1+|Im(z)|\big].
$$

{\bf The boundedness of $\xi\eta\frac{\partial^2 m_z}{\partial\xi\partial\eta}(\xi,\eta)$.}
To deal with $\xi\eta\frac{\partial^2
m_z}{\partial\xi\partial\eta}(\xi,\eta)$, we rewrite it as
\begin{eqnarray*}
\xi\eta\frac{\partial^2
m_z}{\partial\xi\partial\eta}(\xi,\eta)&=&-4\pi^2\xi\eta
\int_{\mathbb R}e^{-2\pi i[\xi
t+\eta\gamma(t)]}\gamma(t)\big[1+\eta^2\gamma^2(t)\big]^{z}dt\\
&-&4\pi iz\xi\eta \int_{\mathbb R}e^{-2\pi i[\xi
t+\eta\gamma(t)]}\big[1+\eta^2\gamma^2(t)\big]^{z-1}\eta\gamma^2(t)dt.
\end{eqnarray*}
For the first term, integrating by parts, we obtain
\begin{eqnarray*}
&&4\pi^2\xi\eta \int_{\mathbb R}e^{-2\pi i[\xi
t+\eta\gamma(t)]}\gamma(t)\big[1+\eta^2\gamma^2(t)\big]^{z}dt\\
&=&2\pi
i\int_{\mathbb R}\frac d{dt}\big(e^{-2\pi i\xi t}\big)e^{-2\pi
i\eta\gamma(t)}[\eta \gamma(t)]\big[1+\eta^2\gamma^2(t)\big]^{z}dt\\
&=&2\pi ie^{-2\pi i[\xi t+\eta\gamma(t)]}[\eta
\gamma(t)]\big[1+\eta^2\gamma^2(t)\big]^{z}\bigg|_{-\infty}^\infty\\
&-&4\pi^2\int_{\mathbb
R}e^{-2\pi i[\xi t+\eta\gamma(t)]}\eta\gamma'(t)[\eta
\gamma(t)]\big[1+\eta^2\gamma^2(t)\big]^{z}dt\\
&-&2\pi i\int_{\mathbb R}e^{-2\pi i[\xi
t+\eta\gamma(t)]}\eta\gamma'(t)\big[1+\eta^2\gamma^2(t)\big]^{z}dt\\
&-&4\pi
iz\int_{\mathbb R}e^{-2\pi i[\xi
t+\eta\gamma(t)]}\big[1+\eta^2\gamma^2(t)\big]^{z-1}\eta^3\gamma^2(t)\gamma'(t)dt.
\end{eqnarray*}
\par
Obviously, for $Re(z)<-1$, $t\in\mathbb R$, $\big|2\pi ie^{-2\pi
i[\xi t+\eta\gamma(t)]}[\eta
\gamma(t)]\big[1+\eta^2\gamma^2(t)\big]^{z}\big|\le2\pi|\eta||\gamma(t)|\big[1+\eta^2\gamma^2(t)\big]^{Re(z)}\le
2\pi$. So, the boundary terms are bounded by $2\pi$.
\par
For the first integrated term, making the change of variables
$u=\eta^2\gamma^2(t)$, we have
\begin{eqnarray*}
&&\bigg|\int_{\mathbb R}e^{-2\pi i[\xi
t+\eta\gamma(t)]}\eta\gamma'(t)[\eta
\gamma(t)]\big[1+\eta^2\gamma^2(t)\big]^{z}dt\bigg|\\
&\le&2\int_{0}^\infty\big[1+\eta^2\gamma^2(t)\big]^{Re(z)}\eta^2
\gamma(t)\gamma'(t)dt\\
&\le&\int_0^\infty\big(1+u\big)^{Re(z)}du\le \frac1{|Re(z)+1|}.
\end{eqnarray*}
\par
The second integrated terms can be treated in the same way, let
$u=\eta\gamma(t)$,
$$
\bigg|\int_{\mathbb R}e^{-2\pi i[\xi
t+\eta\gamma(t)]}\eta\gamma'(t)\big[1+\eta^2\gamma^2(t)\big]^{z}dt\bigg|
\le\int_{\mathbb R}(1+u^2)^{Re(z)}du\le \pi.
$$
\par
Similarly, a trivial calculation shows that
\begin{eqnarray*}
&&\bigg|z\int_{\mathbb R}e^{-2\pi i[\xi
t+\eta\gamma(t)]}\big[1+\eta^2\gamma^2(t)\big]^{z-1}\eta^3\gamma^2(t)\gamma'(t)dt\bigg|\\
&\le&2|z|\int_0^\infty u^2\big(1+u^2\big)^{Re(z)-1}du\le\pi|z|.
\end{eqnarray*}
The second term can be handled similarly. Integrating by parts, we
decompose it as
\begin{eqnarray*}
&&4\pi iz\xi\eta \int_{\mathbb R}e^{-2\pi i[\xi
t+\eta\gamma(t)]}\big[1+\eta^2\gamma^2(t)\big]^{z-1}\eta\gamma^2(t)dt\\
&=&2\pi iz\int_{\mathbb R}\frac d{dt}\big(e^{-2\pi i\xi
t}\big)e^{-2\pi
i\eta\gamma(t)}\eta^2 \gamma^2(t)\big[1+\eta^2\gamma^2(t)\big]^{z-1}dt\\
&=&2\pi ize^{-2\pi i[\xi t+\eta\gamma(t)]}\eta^2
\gamma^2(t)\big[1+\eta^2\gamma^2(t)\big]^{z-1}\bigg|_{-\infty}^\infty\\
\end{eqnarray*}
\begin{eqnarray*}
&-&4\pi^2z\int_{\mathbb R}e^{-2\pi i[\xi
t+\eta\gamma(t)]}\eta\gamma'(t)\eta^2
\gamma^2(t)\big[1+\eta^2\gamma^2(t)\big]^{z-1}dt\\
&-&4\pi iz\int_{\mathbb R}e^{-2\pi i[\xi
t+\eta\gamma(t)]}\eta^2\gamma(t)\gamma'(t)\big[1+\eta^2\gamma^2(t)\big]^{z-1}dt\\
&-&4\pi iz(z-1)\int_{\mathbb R}e^{-2\pi i[\xi
t+\eta\gamma(t)]}\eta^2
\gamma^2(t)\big[1+\eta^2\gamma^2(t)\big]^{z-2}\eta^2\gamma(t)\gamma'(t)dt.
\end{eqnarray*}
\par
Obviously, for $Re(z)<-1$, $t\in\mathbb R$, $ \big|ze^{-2\pi i[\xi
t+\eta\gamma(t)]}\eta^2
\gamma^2(t)\big[1+\eta^2\gamma^2(t)\big]^{z-1}\big|\le |z|$. The
boundary terms are dominated by $4\pi|z|$.
\par
For the first integrated term, by making the change of variables
$u=\eta\gamma(t)$, we have the estimate
\begin{align*}
\bigg|z\int_{\mathbb R}e^{-2\pi i[\xi
t+\eta\gamma(t)]}\eta\gamma'(t)\eta^2
\gamma^2(t)\big[1+\eta^2\gamma^2(t)\big]^{z-1}dt\bigg|
&\le|z|\int_{\mathbb R}u^2(1+u^2)^{Re(z)-1}dt\\
 &\le\pi|z|.
\end{align*}
\par
To estimate the second integrated terms, we make the transformation
$u=\eta^2\gamma^2(t)$ and get
\begin{align*}
\bigg|z\int_{\mathbb R}e^{-2\pi i[\xi
t+\eta\gamma(t)]}\eta^2\gamma(t)\gamma'(t)\big[1+\eta^2\gamma^2(t)\big]^{z-1}dt\bigg|
&\le|z|\int_0^\infty(1+u)^{Re(z)-1}du\\
&\le \frac{|z|}{|Re(z)|}.
\end{align*}
\par
Similarly, the third integrated terms can be treated as
\begin{eqnarray*}
&&\bigg|z(z-1)\int_{\mathbb R}e^{-2\pi i[\xi
t+\eta\gamma(t)]}\big[1+\eta^2\gamma^2(t)\big]^{z-2}\eta^4\gamma^3(t)\gamma'(t)dt\bigg|\\
&\le&|z(z-1)|\int_0^\infty
(1+u)^{Re(z)-1}du\le\frac{|z(z-1)|}{|Re(z)|}.
\end{eqnarray*}
Note that for $Re(z)<-1$, we have the following elementary estimates
$$
|z|\le |Re(z)|\big[1+|Im(z)|\big]\ \ \text{and}\ \ |z-1|\le
|Re(z)-1|\big[1+|Im(z)|\big].
$$
Finally, combining the above eight estimates, we obtain
$$
\big|\xi\eta\frac{\partial^2
m_z}{\partial\xi\partial\eta}(\xi,\eta)\big|\le
C\big[1+Im(z)\big]^2.
$$
This completes the proof of Theorem \ref{main result 2}.
\bigskip

{\bf Acknowledgement.}\ The first author is supported in part by MINECO: ICMAT Severo Ochoa project SEV-2011-0087 and ERC Grant StG-256997-CZOSQP (EU); The second author is supported in part by NSFC 11371057 and 11471033. The authors would like to thank the referee for many valuable and useful comments and suggestions which have improved this paper.

\bibliographystyle{amsplain}

\end{document}